\documentclass[12pt]{amsart}

\usepackage[T1]{fontenc}
\usepackage{amssymb}
\usepackage{enumerate}
\usepackage{hyperref}
\usepackage{color}
\usepackage{graphicx}
\usepackage{amsmath}
\usepackage{amsfonts}
\usepackage{easyReview}
\usepackage{mathtools}

\input xy
\xyoption{all}

\newcommand{\e}{\varepsilon}
\newcommand{\N}{\mathbb{N}}
\newcommand{\R}{\mathbb{R}}

\newcommand{\G}{\mathcal{G}}
\renewcommand{\H}{\mathcal{H}}
\renewcommand{\S}{\mathcal{S}}
\newcommand{\A}{\mathcal{A}}
\renewcommand{\phi}{\varphi}
\newcommand{\w}{\omega}
\newcommand{\ub}{\textbf{u}}

\newcommand{\Lip}{\operatorname{Lip}}

\newcommand{\subs}{\subseteq}
\newcommand{\Gsim}{{\overset{\mathcal{G}}{\sim }}}

\newcommand{\map}[3]{#1\colon #2 \to #3} 
\newcommand{\cmp}{\circ} 
\newcommand{\rest}{\restriction}

\newtheorem{theorem}{Theorem}[section]
\newtheorem{proposition}[theorem]{Proposition}
\newtheorem{lemma}[theorem]{Lemma}
\newtheorem{claim}{Claim}
\newtheorem{corollary}[theorem]{Corollary}
\theoremstyle{definition}
\newtheorem{definition}[theorem]{Definition}
\newtheorem{remark}[theorem]{Remark}

\title[On the connectivity of graph Lipscomb's space]{
	On the connectivity of graph Lipscomb's space
}

\author[W. Kubi\'{s}, R. Miculescu, A. Mihail and M. Nowak]{Wies\l aw KUBI\'{S}, Radu MICULESCU, Alexandru MIHAIL and Magdalena NOWAK}

\address{W. Kubi\'{s}:
	Institute of Mathematics,
	Czech Academy of Sciences,
	\v{Z}itn\'{a} 25, 115 67 Prague, Czech Republic
	\slash \ Department of Mathematics, Cardinal Stefan Wyszy\'nski University in Warsaw, W\'oycickiego 1/3, 01-938 Warszawa}
\email{kubis@math.cas.cz}

\address{R. Miculescu:
	Faculty of Mathematics and Computer Science,
	Transilvania University of Bra\c{s}ov,
	Iuliu Maniu Street, nr. 50, 500091, Bra\c{s}ov, Romania}
\email{radu.miculescu@unitbv.ro}

\address{A. Mihail:
	Faculty of Mathematics and Computer Science,
	University of Bucharest,
	Academiei Street 14, 010014 Bucharest, Romania}
\email{mihail\_alex@yahoo.com}

\address{M. Nowak
	(ORCID 0000-0003-1915-0001): 
	Mathematics Department, 
	Jan Kochanowski University in Kielce,
	Uniwersytecka 7, 25-406 Kielce, Poland}
\email{magdalena.nowak805@gmail.com}

\date{\today}

\begin{document}
	
	\begin{abstract}
		A central role in topological dimension theory is played by Lipscomb's space $J_{A}$ since it is a universal space for  metric spaces of weight $|A|\geq \aleph _{0}$. On the one hand, Lipscomb's space is the attractor of a possibly infinite iterated function system, i.e. it is a generalized Hutchinson-Barnsley fractal. As, on the other hand, some classical fractal sets are universal spaces, one can conclude that there exists a strong connection between topological dimension theory and fractal set theory. A generalization of Lipscomb's space, using graphs, has been recently introduced (see R. Miculescu, A. Mihail, Graph Lipscomb's space is a generalized Hutchinson-Barnsley fractal, Aequat. Math., \textbf{96} (2022), 1141-1157). It is denoted by $J_{A}^{\G}$ and it is called graph Lipscomb's space associated with the graph $\G$ on the set $A$. It turns out that it is a topological copy of a generalized Hutchinson-Barnsley fractal. This paper provides a characterization of those graphs $\G$ for which $J_{A}^{\G}$ is connected. In the particular case when $A$ is finite, some supplementary characterizations are presented.
	\end{abstract}
	\maketitle

\section{Introduction}

Given a complete metric space $(X,d)$ and a finite family $(f_{i})_{i\in I}$
of contractions on $X$, a classical result of J. Hutchinson (see \cite{Hu}) states
that there exists a unique nonempty compact subset $A$ of $X$ such that $A=%
\bigcup_{i\in I}f_{i}(A)$. It is called the \emph{attractor} of the
iterated function system $\mathcal{S}=((X,d),(f_{i})_{i\in I})$. Since this
framework (which was honed and popularized by M. Barnsley, see \cite{B})
generates almost all the classical fractals, the attractors of iterated
function systems are also called Hutchinson-Barnsley fractals (briefly, H-B
fractals). One way to extend the class of H-B fractals is to consider
systems comprising an infinite family of contractions (see \cite{F,GJ,MU,M,MM_shift}). This way one obtains a larger class of H-B fractals, called
generalized Hutchinson-Barnsley fractals, which are non-empty closed and
bounded subsets of $X$. Note that on the one hand, there are nonempty
compact subsets of $\mathbb{R}$ which are not H-B fractals. On the other
hand, each non-empty compact subset of $X$ is a generalized H-B fractal.

The study of topological features (containing, but not bounded to, connectedness) of (generalized) H-B fractals is a central issue in fractal
theory. For example, M. Hata (see \cite{Ha}) provided conditions under which such
fractals have certain connectivity properties. As an application we mention
a characterization of a compact operator via the non-connectedness of some
H-B fractals related to the given operator (see \cite{MM_char}). For the particular
case of the connectedness of tiles (which play a chief role not only in
fractal theory, but also in wavelet theory) see the very recent papers \cite{D,WL},
and the references therein.

Lipscomb's space $J_{A}$ was introduced in 1975 by S.L. Lipscomb (see
\cite{L_embed}) as a universal space for metric spaces having weight $|A|\geq \aleph
_{0}$. For details and for the connection between topological dimension
theory and generalized Hutchinson-Barnsley fractals we refer to~\cite{L_fract}. The concept of
graph Lipscomb's space $J_{A}^{\mathcal{G}}$ associated with a graph $\mathcal{G}$ on the set $A$, which is a generalization of the classical
notion of Lipscomb's space $J_{A}$, has been recently introduced in \cite{MM_graph}. It
turned out that its embedded version in $\ell^{2}(A')$, denoted by $\omega _{\mathcal{G}}^{A}$, where $A'=A\setminus \{z\}$, $z$
being a fixed element of $A$, is a generalized H-B fractal. As it is known
that the embedded version of $J_{A}$ in $\ell^{2}(A')$ is connected
(see \cite{DM}), a natural problem arises: find necessary and sufficient
conditions on $\mathcal{G}$ for the connectedness of the embedded version of 
$J_{A}^{\mathcal{G}}$ in $\ell^{2}(A')$. In this paper we prove
that $\omega _{\mathcal{G}}^{A}$ is connected if and only if the graph $\mathcal{G}$ is connected, if and only if $J^\G_A$ is connected. In case when $A$ is finite, we also prove that if $\G$ is connected, then $J^\G_A$  (and consequently $\omega_\G^A$) is locally connected and arcwise connected.

\section{Graph Lipscomb's space}

First we present some basics of graph theory. By set $\mathbb{N}$ of natural numbers we mean $\{1,2,...\}$.

\begin{definition}
	By a \textit{graph} we understand a pair $\G\coloneqq(V,E)$ consisting the set $V\neq\emptyset$ of \textit{vertices} and the set $E\subs \{\{v,u\}: v,u\in V\}$ of \textit{edges} such that $\{v\}\in E$ for every vertex $v\in V$. 
\end{definition}
In graph theory such an object is precisely called \textit{undirected simple graph with loops in every vertex}. For two vertices $v,u\in V$ we say that $v$ is \textit{adjacent to} $u$ when $\{v,u\}\in E$. In our framework, the edges of a given graph $\G$ induce a reflexive and symmetric binary relation on vertices called the \textit{adjacency relation} on $\G$. 

\begin{definition}
	A graph $\G=(V,E)$ is called \textit{connected} if for every $v,u\in V$ there exist a positive natural number $n$ and a set $\{v_1,\dots,v_n\} \subs V$ such that $v=v_1, v_n=u$ and $\{v_k,v_{k+1}\}\in E$ for every $k\in\{1,\dots,n-1\}$. 
\end{definition}

 Let $A$ be a set with at least two elements and $\G=(A,E)$  be a graph on $A$. By $N(A)$ we denote the Baire space $(A^\N,d)$ with the metric 

$$d(x,y)=
\begin{cases}
	0 & \text{ if }x=y,\\
	\frac{1}{\min\{n\in\N\colon x_n\neq y_n\}} & \text{ if }x\neq y
\end{cases}
$$
for words $x=x_1x_2..., y=y_1y_2... \in A^\N$.

\begin{definition}
	The \textit{graph Lipscomb's space} $J_A^\G$ is the quotient space $N(A)/_{\Gsim}$ for the following equivalence relation $\Gsim$ on $N(A)$:
	$$x~\Gsim~ y \quad \text{ if } \quad x=y \text{ or }x = sa\vec b,  y= sb\vec a \text{ such that } \{a,b\}\in E,$$
	where $s$ is a finite word of symbols from $A$ (may by empty) and $\vec a=aaa\dots, \vec b=bbb\dots$ are infinite words with constant $a$ or $b$, respectively.
\end{definition}
As it is shown in \cite{MM_graph}, every equivalence class has at most two elements.

Let us fix $z\in A$ and denote $A'\coloneqq A\setminus\{z\}$. Recall that 
$$\ell^{2}(A')=\Big\{(x_{j})_{j\in A'}\in\R^{A'}\colon x_{j}\neq 0\text{ for only countably many } j\in A' \text{ and }\sum_{j\in A'}x_{j}^{2}<\infty\Big\}$$
is the canonical Hilbert space with the Euclidean metric 
$d_e(x,y)=\sqrt{\sum_{b\in A'}(x_b-y_b)^2}$
for every $x=(x_b)_{b\in A'}$ and $y=(y_b)_{b\in A'}$.

\begin{definition}
Define the following objects:
\begin{itemize}
	\item the function $\pi\colon N(A)\to N(A)/_\Gsim=J^A_\G$ is the quotient, canonical, projection map
	\item the map $p_\G\colon N(A)\to \ell^2(A')$ is defined by $p_\G(x)=(\alpha_b(x))_{b\in A'}$ for every $x=x_1x_2... \in N(A)$, where
	$$\alpha_b(x)\coloneqq \sum_{i=1}^{\infty}	\prod_{j=1}^{i}c_{x_jx_ib}
	\quad\text{ and }\quad
	c_{aa'b}\coloneqq
	\begin{cases}
		0 & \text{ if } a'\neq b\\
		\frac{1}2 & \text{ if } a'=b \wedge \{a,a'\}\in E \\
		\frac{1}3 & \text{ if } a'=b \wedge \{a,a'\}\notin E 
	\end{cases}\quad
	\text{ for any } a,a',b\in A.$$
	\item the space $\w^A_\G\coloneqq p_\G(N(A))$ is a subset of $\ell^2(A')$
	\item the function $s_\G\colon J_A^\G\to \w^A_\G$ is defined by $s_\G([x])=p_\G(x)$, where $[x]$ designates the equivalence class containing $x\in N(A)$.  It is is well defined bijection cause $p_\G(x)=p_\G(y)$ iff $[x]=[y]$ (see \cite[Proposition 3.6]{MM_graph}). 
\end{itemize}
\end{definition}

Thus we have that $p_\G = s_\G\circ \pi$ and the following diagram
$$\xymatrix{
	N(A) \ar[r]^-{\pi} \ar[dr]^-{p_\G} & J_A^\G \ar[d]^-{s_\G}  \\
	& \w^A_\G. }$$
Moreover, for every $x\in N(A)$ the values $s_\G([x])$ and $p_\G(x)$ can be written as follows 
$$s_\G([x])=p_\G(x)=\sum_{i=1}^{\infty}	\Big(\prod_{j=1}^{i}c_{x_jx_ix_i}\Big)\ub_{x_i},$$
where $\ub_{z}$ is the zero element and $(\ub_a)_{a\in A'}$ is the canonical basis of the vector space $\ell^2(A')$ i.e. $\ub_a=(\delta _{ab})_{b\in A'}\in\ell^{2}(A')$ is described by  $\delta _{ab}=
\begin{cases}
	1& \text{ if } a=b;\\ 
	0& \text{ if }a\neq b.
\end{cases}
$
\newline

The concept of graph Lipscomb's space introduced by R. Miculescu and A. Mihail is the generalization of the theory considered in the book \cite{L_fract}. The classical Lipscomb's space $J_A$ is a graph Lipscomb's space for so called complete graphs (where every two vertex are adjacent). We extend results obtained by Lipscomb to show that the map $s_\G$ is a homeomorphism so the space $\w_\G^A$ is homeomorphic to the graph Lipscomb's space $J_A^\G$. First we note that

\begin{remark}\label{rem_c}
	For a given $x\in N(A)$ and $i\in\N$ we have
	$$\sum_{b\in A'}\prod_{j=1}^{i}c_{x_{j}x_{i}b} = \prod_{j=1}^{i}c_{x_{j}x_{i}x_{i}} \in\{0\}\cup \Big[ \frac{1}{3^{i-1}\cdot2}, \frac{1}{2^i}\Big].$$
\end{remark}

\begin{lemma}\label{lem_p_cont}
	The map $p_\G\colon N(A)\to \ell^2(A')$ is continuous.
\end{lemma}
\begin{proof} Take the sequence $(x^n)_{n\in\N}$ of words $x^n= x^n_1 x^n_2...\in N(A)$ convergent to $x= x_1x_2...\in N(A)$. Then for every natural $k\in\N$ there exists $n_k\in\N$ such that for $n\geq n_k$ we have $x^n_i=x_i$ for every $i=1,\dots,k$ and
\begin{align*}
		d_e(p_\G(x^n), p_\G(x))^2 & =\sum_{b\in A'}(\alpha_b(x^n)-\alpha_b(x))^2\leq \sum_{b\in A'}| \alpha_b(x^n)-\alpha_b(x) |=\\
		&=\sum_{b\in A'}\Big|\sum_{i=k+1}^\infty(\prod_{j=1}^{i}c_{x^n_j x^n_i b}-\prod_{j=1}^{i}c_{x_{j}x_{i}b})\Big|\leq\\
		&\leq \sum_{b\in A'}\sum_{i=k+1}^\infty\Big|\prod_{j=1}^{i}c_{x^n_j x^n_i b}-\prod_{j=1}^{i}c_{x_{j}x_{i}b}\Big|=\\ 
		&=\sum_{i=k+1}^\infty\sum_{b\in A'}\Big|\prod_{j=1}^{i}c_{x^n_j x^n_i b}-\prod_{j=1}^{i}c_{x_{j}x_{i}b}\Big|=\\
		&=\sum_{i=k+1}^\infty\Big(|\prod_{j=1}^{i}c_{x^n_j x^n_i x^n_i}-\prod_{j=1}^{i}c_{x_{j}x_{i}x^n_{i}}|+|\prod_{j=1}^{i}c_{x^n_j x^n_i x_{i}}-\prod_{j=1}^{i}c_{x_{j}x_{i}x_{i}}|\Big)\leq\\
		&\leq \sum_{i=k+1}^\infty \Big(\frac1{2^i}+\frac1{2^i}\Big)=\frac1{2^{k-1}}.
	\end{align*}
	This means that $d_e(p_\G(x^n), p_\G(x))\to 0$ which completes the proof. 
\end{proof}

\begin{proposition}\label{PropEng}\cite[Prop.2.4.3.]{Eng}
	For the map $p$ of the topological space $X$ onto topological space $Y$ and for the equivallence relation $\sim$ on $X$ determined by the decomposition $(p^{-1}(y))_{y\in Y}$, the following conditions are equivalent:
	\begin{enumerate}[(i)]
		\item $p$ is a composition of the quotient mapping $\pi\colon X\to X/_\sim$ and the homeomorphism $s\colon X/_\sim\to Y$
		\item the set $U$ is open in $Y$ iff $p^{-1}(U)$ is open in $X$
		\item the set $U$ is closed in $Y$ iff $p^{-1}(U)$ is closed in $X$
		\item the map $s\colon X/_\sim\to Y$ defined by letting $s(p^{-1}(y))=y$, is a homeomorphism.
	\end{enumerate}
\end{proposition}

\begin{theorem}\label{s_hom}
	The map $s_\G\colon J_A^\G\to \w^A_\G$ is a homeomorphism.
\end{theorem}
\begin{proof} 
	According to the Proposition \ref{PropEng} and Lemma \ref{lem_p_cont} we shall show that the set $U\subs \omega^A_\G$ is closed whenever $p_\G^{-1}(U)$ is closed in $N(A)$. Take an arbitrary convergent sequence $(y_n)_{n\in\N}$ in $U$ convergent to some limit $y\in\w^A_\G$. Consider the sequence of sets $Y_n\coloneqq p_\G^{-1}(y_n)\subs p_\G^{-1}(U)$. Then we show that
	\begin{claim}\label{cl_0}
		There exist $x\in N(A)$, strictly increasing sequence $(n_k)_{k\in\N}$ and sequence $(x^{n_k})_{k\in\N}$ of words such that $x^{n_k}\in Y_{n_k}$ and $x= \lim_{k\to\infty}x^{n_k}$. 
	\end{claim}
	\begin{proof}
		Assume, in the contrary, that there is no convergent subsequence in any sequence $(x^n)_{n\in\N}$ of words such that $x^n\in Y_n$ for $n\in\N$. Then for every such sequence there exists $k\geq 1$ such that the set $K\coloneqq\{x^n_k\colon n\in\N\}$ is infinite. [Otherwise we could inductively define a sequence $x\in N(A)$ such that for any $k\geq 1$ the distance $d(x,x^n)<\frac1{k}$ for infinitely many $n$, thus $x$ is a limit of some subsequence of $(x^n)_{n\in\N}$.] \\Moreover we can find the numbers $\e>0$ and natural number $l\geq k$ such that 
		$$\frac1{2^l} +\e < \frac1{3^{k-1}\cdot 2}.$$
		Then for every $x=x_1 x_2...\in p_\G^{-1}(y)$ also the set $K_x:=K\setminus \{x_1, \dots, x_l\}$
		is infinite. Thus, for infinitely many $n\in \N$ the symbol $x^n_k\in K_x$ and for such $b'\coloneqq x^n_k$ we have
		\begin{align*}
			d_e(y_n, y) & =d_e(p_\G(x^n), p_\G(x))  =\sqrt{ \sum_{b\in A'}(\alpha_{b}(x^n)-\alpha_{b}(x))^2} \geq | \alpha_{b'}(x^n)-\alpha_{b'}(x) |=\\
			&=\Big|\sum_{i=1}^\infty(\prod_{j=1}^{i}c_{x^n_j x^n_i b'}-\prod_{j=1}^{i}c_{x_{j}x_{i}b'})\Big|=\\
			&=\Big|\sum_{i=1}^l \prod_{j=1}^{i}c_{x^n_j x^n_i b'} + \sum_{i=l+1}^\infty(\prod_{j=1}^{i}c_{x^n_j x^n_i b'}-\prod_{j=1}^{i}c_{x_{j}x_{i}b'})\Big|\geq\\
			&\geq \frac1{3^{k-1}\cdot 2} -  \sum_{i=l+1}^\infty\frac1{2^i}=\frac1{3^{k-1}\cdot 2} -\frac1{2^{l}}>\e.
		\end{align*}
		This is a contradiction with the fact that $y_n\to y$. 
	\end{proof}   
	Due to the Claim \ref{cl_0}, there exists $x\in N(A)$, strictly increasing sequence $(n_k)_{k\in\N}$ and sequence $(x^{n_k})_{k\in\N}$ of words such that $x^{n_k}\in Y_{n_k}$ and $x= \lim_{k\to\infty}x^{n_k}$. All elements $x^{n_k}$ belong to the closed set $p_\G^{-1}(U)$ so also $x\in p_\G^{-1}(U)$. By the continuity of the map $p_\G$, proved in Lemma \ref{lem_p_cont}, we have
	$$p_\G(x)=\lim_{k\to\infty}p_\G(x^{n_k})=\lim_{k\to\infty} y_{n_k} = y.$$
	Thus $y\in U$, cause $p_\G(p_\G^{-1}(U))=U$. Thus we proved that $U$ is closed, hence $p_\G$ satisfies (iii) from Proposition \ref{PropEng} so the map $s_\G$ is a homeomorphism.
\end{proof}

\section{Connectedness of $J^\G_A$ and $\w^A_\G$}

The main aim of this paper is to characterize connected spaces $J^\G_A$ and $\w^A_\G$. To do this, we will use the fact that the space $\w^A_\G$ is a generalized Hutchinson-Barnsley fractal or, in other words, an attractor of the possibly infinite iterated function system. Therefore we present some basics of the Hutchinson-Barnsley fractal theory.

Let $(X,d)$ be the arbitrary metric space. For a function $f\colon X\to X$ and $n\in \N$, the composition of $f$ by itself $n$ times will be denoted by $f^{[n]}$ and $\Lip( f)$ stands for the Lipschitz constant of $f$. A map $f$ satisfying $\Lip(f)<1$ is called a \textit{contraction}. Moreover

\begin{itemize}
	\item $P_{b}(X) \coloneqq	\{B\subseteq X\mid B\text{ is bounded and nonempty}\}$
	\item $P_{cl}(X) \coloneqq \{B\subseteq X\mid B\text{ is closed and nonempty}\}$
	\item $P_{b,cl}(X)\coloneqq P_{b}(X)\cap P_{cl}(X)$
	\item $P_{cp}(X)\coloneqq \{B\subseteq X\mid B\text{ is compact}\}.$
\end{itemize}

The space $P_{b,cl}(X)$ is a metric space with the standard Hausdorff-Pompeiu
metric $h$, described by 
\begin{equation*}
h(A,B)=\max \{\underset{a\in A}{\sup }\text{ }d(a,B),\underset{b\in B}{\sup }%
\text{ }d(b,A)\},
\end{equation*}
where $d(x,Y)=\inf\{d(x,y)\colon y\in Y\}$ for a set $Y\in P_{b,cl}(X)$ and an element $x\in X$.
The topology generated by the metric $h$ is the same as the standard Vietoris topology on $P_{b,cl}(X)$. 

\begin{definition}
	An\textit{ iterated function system }(IFS for
	short), is a pair $\S\coloneqq ((X,d),(f_{i})_{i\in I})$ consists of:
	
\begin{enumerate}[i)]
	\item a complete metric space $(X,d)$;
	\item a finite family of contractions $f_{i}:X\to X$  for ${i\in I}$.
\end{enumerate}
	The function $F_{\mathcal{S}}:P_{cp}(X)\rightarrow P_{cp}(X)$ defined by 
	$$F_{\S}(K)=\bigcup_{i\in I}f_{i}(K),$$
	for all $K\in P_{cp}(X)$, which is called the \textit{set function associated to} $\mathcal{S}$, turns out to be a contraction (with respect to the Hausdorff-Pompeiu metric) and its unique fixed point, denoted by $A_{\mathcal{S}}$, is called the \textit{attractor} of $\mathcal{S}$.
	Moreover, for every $K\in P_{cp}(X)$
	$$\lim_{n\to\infty}h(F_\S^{[n]}(K),A_\S)=0.$$
\end{definition}

The sets obtained as attractors of IFSs are called Hutchinson-Barnsley
fractals.

\begin{definition}
	(see \cite{MM_shift}) A \textit{possibly infinite iterated function system} (IIFS for short) is a pair $((X,d),(f_{i})_{i\in I})$, denoted by $\mathcal{S}$, which
		consists of:
		\begin{enumerate}[i)]
			\item a complete metric space $(X,d)$;
			\item a family of contractions $f_{i}\colon X\to X$ for $i\in I$, with the following two properties:
			\begin{enumerate}[a)]
				\item $\sup_{i\in I}\Lip(f_i)<1$
				\item the family $(f_{i})_{i\in I}$ is bounded (i.e. $\bigcup_{i\in I} f_{i}(A)\in P_{b}(X)$ for every $A\in P_{b}(X)$, which is equivalent to the fact that there exists $x_{0}\in X$ such that $(f_{i}(x_{0}))_{i\in I}$ is bounded).
			\end{enumerate}
		\end{enumerate}
		The function $F_{\mathcal{S}}\colon P_{b,cl}(X)\to P_{b,cl}(X)$, given by
		$$F_{\mathcal{S}}(B)=\overline{\bigcup_{i\in I} f_{i}(B)},$$
		for all $B\in P_{b,cl}(X)$, which is called the \textit{fractal operator associated to} $\mathcal{S}$, is a contraction and therefore, it has a unique fixed point, denoted by $A_{\mathcal{S}}$, which is called the \textit{attractor of} $\mathcal{S}$.
		
		Moreover,
		$$\lim_{n\to\infty} h(F_{\mathcal{S}}^{[n]}(B),A_{\mathcal{S}})=0,$$
		for every $B\in P_{b,cl}(X)$.
\end{definition}

The sets obtained as attractors of IIFSs are called \textit{generalized Hutchinson-Barnsley fractals}.

\begin{proposition}\label{prop_limcp}
	 \cite[Proposition 2.7]{MM_gen} Let $(X,d)$ be a complete metric space, $B\in P_{b,cl}(X)$ and $(B_{n})_{n\in \mathbb{N}}\subseteq P_{cp}(X)$ such that $lim_{n\to\infty}h(B_{n},B)=0$. Then $B\in P_{cp}(X)$.
\end{proposition}

\begin{remark}
	The concept of IIFS is a generalization of the one of IFS. Indeed, let $((X,d),(f_{i})_{i\in I})$ be an IIFS for which $I$ is finite and let $A_{\mathcal{S}}$ be its fixed point. Then $F_{\mathcal{S}}(K)=\overline{\bigcup_{i\in I} f_{i}(K)}=\bigcup_{i\in I} f_{i}(K)\in P_{cp}(X)$ and, inductively, we have $F_{\mathcal{S}}^{[n]}(K)\in P_{cp}(X)$ for every $K\in P_{cp}(X)$ and $n\in \mathbb{N}$.
	As $\lim_{n\to\infty} h(F_{\mathcal{S}}^{[n]}(K),A_{\mathcal{S}})=0$, Proposition \ref{prop_limcp} yields $A_{\mathcal{S}}\in P_{cp}(X)$, i.e. $A_{\mathcal{S}}$ is the attractor of the IFS $((X,d),(f_{i})_{i\in I})$.
\end{remark}

\begin{definition}
	Let $(X,d)$ be a metric space and $(A_{i})_{i\in I}$ a family of nonempty subsets of $X$.
	The family $(A_{i})_{i\in I}$ is said to be \textit{connected} if for every $i,j\in I$ there exist $n\in \mathbb{N}$ and $\{i_{1},...,i_{n}\}\subseteq I$ such that $i_{1}=i$, $
	i_{n}=j$ and $A_{i_{k}}\cap A_{i_{k+1}}\neq \emptyset $
	for every $k\in \{1,...,n-1\}$.
\end{definition}

\begin{theorem}\label{thm_atcon}
	\cite[Theorem 1.6.2, page 33]{K}
	Given an IFS $\mathcal{S}=((X,d),(f_{i})_{i\in I})$\, the
	following statements are equivalent:
	\begin{enumerate}[1)] 
		\item the family $(f_{i}(A_{\mathcal{S}}))_{i\in I}$ is connected;
		\item the attractor $A_{\mathcal{S}}$ is connected.
	\end{enumerate}
\end{theorem}

For the sake of brevity, in the sequel, for the elements of $\ell^2(A')$ we shall sometimes use the notation $(x_{j})$ instead $(x_{j})_{j\in A'}$.

\begin{definition}
	Given a graph $\mathcal{G}$ on $A$, one can consider the IIFS 
	\begin{equation*}
	\mathcal{S}_{\mathcal{G}}^{A}\coloneqq ((\ell^{2}(A'), d_e ),(f_{i})_{i\in A})
	\end{equation*}
	where $f_{i}\colon\ell^{2}(A')\to \ell^{2}(A')$ is given by 
	$f_{i}((x_{k}))=(c_{ikk}x_{k})_{k\in A'}+\frac{1}{2}\mathbf{u}_{i}$ for every 
	$(x_{k})\in \ell^{2}(A')$ and every $i\in A$.
\end{definition}

In the above framework, the main result from \cite{MM_graph} states that the space $\w_\G^A$ is the attractor of IIFS above:
\begin{equation*}
\omega _{\mathcal{G}}^{A}=A_{\mathcal{S}_{\mathcal{G}}^{A}},
\end{equation*}
i.e. the embedded version of $J_{A}^{\mathcal{G}}$ on $\ell^{2}(A')$
is a generalized Hutchinson-Barnsley fractal. In particular, considering the complete graph $\mathcal{CG}=(A,\{\{a,b\}\}_{a,b\in A})$ and the classical Lipscomb space, one obtains the main result of \cite{MM_lips}. In addition, $\omega _{\mathcal{CG}}^{A}$ is connected (see \cite{DM}), so a natural idea is to find necessary and sufficient conditions for the connectedness of 
$\omega _{\mathcal{G}}^{A}$, for arbitrary graph $\G$.  The main result of this paper states that

\begin{theorem}\label{main}
	The following statements are equivalent: 
	\begin{enumerate}[(i)]
		\item the space $J_{A}^{\mathcal{G}}$ is connected
		\item the space $\omega _{\mathcal{G}}^{A}$  is connected
		\item the graph $\G=(A,E)$ is a connected graph.
	\end{enumerate}
\end{theorem}

The equivalence $(i)\Leftrightarrow(ii)$ follows from the fact proved in Theorem \ref{s_hom}, that spaces $J_{A}^{\mathcal{G}}$ and $\omega _{\mathcal{G}}^{A}$ are homeomorphic. We justify implications $(iii)\Rightarrow(ii)$ and $(i)\Rightarrow(iii)$  by proving Propositions \ref{<=} and \ref{=>}.

\begin{proposition}\label{<=}
	If $\mathcal{G}=(A,E)$ is a connected graph, then $\omega _{\mathcal{G}}^{A}$ is connected.
\end{proposition}

\begin{proof}
	Let us start by introducing some notation. For $i,j\in A$ such that $\{i,j\}\in E$, we consider the iterated function system 
	\begin{equation*}
		((\ell^{2}(A'), d_e),\{f_{i},f_{j}\}):=\mathcal{S}_{\{i,j\}}.
	\end{equation*}
	
	We shall also consider
	$$\A\coloneqq \bigcup_{\{i,j\}\in E} A_{\S_{\{i,j\}}}$$
	and
	\begin{equation*}
		\mathbf{A}\coloneqq\overline{\A}.
	\end{equation*}
	Since $A_{\mathcal{S}_{\{i,j\}}}\subseteq A_{\mathcal{S}_{\mathcal{G}}^{A}}$
	for every $\{i,j\}\in E$ (see \cite[Theorem 4.1]{MM_shift}) we infer that $\mathcal{%
		A\subseteq }A_{\mathcal{S}_{\mathcal{G}}^{A}}$, so $\mathbf{A}=\overline{%
		\mathcal{A}}\mathcal{\subseteq }A_{\mathcal{S}_{\mathcal{G}}^{A}}$.
	Therefore, as $A_{\mathcal{S}_{\mathcal{G}}^{A}}\in
	P_{b,cl}(\ell^{2}(A'))$, we conclude that $\textbf{A}\in
	P_{b,cl}(\ell^{2}(A'))$ and, consequently, we can consider the
	sequence $(A_{n})_{n\in \mathbb{N\cup }\{0\}}\subseteq
	P_{b,cl}(\ell^{2}(A'))$ given by
	$$A_{0}\coloneqq\mathbf{A} \quad\text{ and } \quad A_{n+1}\coloneqq F_{\mathcal{S}_{\mathcal{G}}^{A}}(A_{n}),$$
	for every $n\in \mathbb{N\cup }\{0\}$.
	
	\begin{claim}\label{cl_1}
		$A_{\mathcal{S}_{\{i,j\}}}$ is connected for every $\{i,j\}\in E$.
	\end{claim}
	
	\begin{proof}[Proof of Claim \ref{cl_1}]
		Note that for $i,j\in A$ we have
		\begin{equation*}
			f_{i}(\mathbf{u}_{i})=(c_{ikk}\delta _{ik})_{k\in A'}+\frac{1}{2}%
			\mathbf{u}_{i}=\frac{1}{2}\mathbf{u}_{i}+\frac{1}{2}\mathbf{u}_{i}=\mathbf{u}%
			_{i}
		\end{equation*}
		and, similarly,
		\begin{equation*}
			f_{j}(\mathbf{u}_{j})=\mathbf{u}_{j}\text{,}
		\end{equation*}%
		i.e. $\mathbf{u}_{i}$ is the fixed point of $f_{i}$ and $\mathbf{u}_{j}$ is
		the fixed point of $f_{j}$. Consequently 
		\begin{equation*}
			\mathbf{u}_{i},\mathbf{u}_{j}\in A_{\mathcal{S}_{\{i,j\}}}\text{.}  
		\end{equation*}
		Since  $\{i,j\}\in E$ and
		\begin{equation*}
			f_{i}(\mathbf{u}_{j})=(c_{ikk}\delta _{jk})_{k\in A'}+\frac{1}{2}%
			\mathbf{u}_{i}=c_{ijj}\mathbf{u}_{j}+\frac{1}{2}\mathbf{u}_{i}=\frac{1}{2}\mathbf{u}_{j}+\frac{1}{2}\mathbf{u}_{i}=f_{j}(%
			\mathbf{u}_{i})\in f_{i}(A_{\mathcal{S}%
				_{\{i,j\}}})\cap f_{j}(A_{\mathcal{S}_{\{i,j\}}})\text{,}
		\end{equation*}
		we infer that
		\begin{equation*}
			f_{i}(A_{\mathcal{S}_{\{i,j\}}})\cap f_{j}(A_{\mathcal{S}_{\{i,j\}}})\neq
			\emptyset \text{.}
		\end{equation*}%
		We conclude, via Theorem \ref{thm_atcon}, that $A_{\mathcal{S}_{\{i,j\}}}$ is connected. Note also that all reasoning and calculations are true also in the case $i=j$. Then $A_{\S_{\{i\}}}=\{\ub_i\}$. 
	\end{proof}

	\begin{claim}\label{cl_2}
		The family $(A_{\mathcal{S}_{\{i,j\}}})_{\{i,j\}\in E}$ is connected.
	\end{claim}
	
	\begin{proof}[Proof of Claim \ref{cl_2}.]
		Let us consider $\{i,j\},\{k,l\}\in E$. Then, as $\mathcal{G}$ is connected, there exist $n\in 
		\mathbb{N}$ and $\{j_{1},...,j_{n}\}\subseteq A$ such that $j_{1}=j$, $%
		j_{n}=k$ and 
		\begin{equation*}
			\{j_{s},j_{s+1}\}\in E,
		\end{equation*}
		for every $s\in \{1,...,n-1\}$. Since, thanks of \cite[Theorem 4.1]{MM_shift},  $\emptyset \neq A_{\S_{\{t\}}}\subs A_{\S_{\{t,t'\}}}$ for every $t,t'\in A$, we infer that 
		\begin{align*}
			A_{\S_{\{j\}}}=A_{\S_{\{j_1\}}}\subs A_{\S_{\{i,j\}}}\cap A_{\S_{\{j_1,j_2\}}}\neq\emptyset &, \\
			A_{\S_{\{j_s\}}}\subs A_{\S_{\{j_{s-1},j_{s}\}}}\cap A_{\S_{\{j_s,j_{s+1}\}}}\neq\emptyset & \text{ for every }s=2,...,n-1,\\
			A_{\S_{\{k\}}}=A_{\S_{\{j_n\}}}\subs A_{\S_{\{j_{n-1},j_n\}}}\cap A_{\S_{\{k,l\}}}\neq\emptyset &,
		\end{align*}
		which means that the family $(A_{\mathcal{S}_{\{i,j\}}})_{\{i,j\}\in E}$ is connected.
	\end{proof}

	\begin{claim}\label{cl_3}
		The set $\mathbf{A}=\overline{\A}=\overline{\bigcup_{\{i,j\}\in E} A_{\S_{\{i,j\}}}}$ is connected.
	\end{claim}
	
	\begin{proof}[Proof of Claim \ref{cl_3}.]
		It follows, via \cite[Theorem 1.5 page 108 and Problem 6, page 117]{D}, from Claim \ref{cl_1} and Claim \ref{cl_2} that $\mathcal{A}$ is connected, so $\overline{\mathcal{A}}=\mathbf{A}$ is connected.
	\end{proof}
	
	\begin{claim}\label{cl_4}
		We have $A_{n}\subseteq A_{n+1}$  for every $n\in \mathbb{N}\cup \{0\}$.
	\end{claim}
	
	\begin{proof}[Proof of Claim \ref{cl_4}.]
		Since 
		\begin{equation*}
			A_{\mathcal{S}_{\{i,j\}}}=F_{\mathcal{S}_{\{i,j\}}}(A_{\mathcal{S}%
				_{\{i,j\}}})\subseteq F_{\mathcal{S}_{\mathcal{G}}^{A}}(A_{\mathcal{S}%
				_{\{i,j\}}})\text{,}
		\end{equation*}%
		for every $\{i,j\}\in E$, we obtain%
		\begin{equation*}
			\mathcal{A}=\bigcup_ {\{i,j\}\in E}A_{\mathcal{S}_{\{i,j\}}}\subseteq
			F_{\mathcal{S}_{\mathcal{G}}^{A}}\Big(\bigcup_ {\{i,j\}\in E}A_{\mathcal{S%
				}_{\{i,j\}}}\Big)=F_{\mathcal{S}_{\mathcal{G}}^{A}}(\mathcal{A})\subseteq F_{%
				\mathcal{S}_{\mathcal{G}}^{A}}(A_{0})=A_{1}\text{,}
		\end{equation*}%
		so%
		\begin{equation*}
			\overline{\mathcal{A}}=\mathbf{A}=A_{0}\subseteq A_{1}\text{.}
		\end{equation*}
		By mathematical induction we easily conclude that $A_{n}\subseteq A_{n+1}$
		for every $n\in \mathbb{N\cup }\{0\}$.
	\end{proof}
	
	\begin{claim}\label{cl_5}
		The set $A_{n}$ is connected for every $n\in \mathbb{N\cup }\{0\}$.
	\end{claim}
	
	\begin{proof}[Proof of Claim \ref{cl_5}.]
		Claim \ref{cl_3} ensures that $A_{0}$ is connected.
		
		Now suppose that $A_{n}$ is connected. We will prove that $A_{n+1}$ is connected.
		Note that since
		\begin{equation*}
			A_{n+1}=F_{\mathcal{S}_{\mathcal{G}}^{A}}(A_{n})=\overline{\bigcup_{i\in A}f_{i}(A_{n})}\overset{\text{Claim \ref{cl_4}}}{=}\overline{A_{0}\cup \Big(\bigcup_{i\in A}f_{i}(A_{n})\Big)},
		\end{equation*}
		it suffices to prove that $A_{0}\cup \big(\bigcup_{i\in A}f_{i}(A_{n})\big)$ is connected.
		
		Let us note that
		\begin{equation}\label{eq3}
			A_{0}\cap f_{i}(A_{n})\neq \emptyset ,
		\end{equation}%
		for every $i\in A$. 
		Indeed, by the fact that 
			$A_{\S_{\{i\}}}\subs A_0\overset{\text{Claim \ref{cl_4}}}\subs A_n$, we have  
				$f_i(A_{\S_{\{i\}}})\subs f_i(A_n)$ and $f_i(A_{\S_{\{i\}}})=A_{\S_{\{i\}}}\subs A_0$.
		
		The set $A_{0}$ is connected (see Claim \ref{cl_3}) and, based on the continuity of $f_{i}$
		and on the connectedness of $A_{n}$, the set $f_{i}(A_{n})$ is connected,
		for every $i\in A$.
		
		Via \cite[Problem 6, page 117]{D}, the property (\ref{eq3}) ensures that $A_{0}\cup \big(\bigcup_{i\in A}f_{i}(A_{n})\big)$ is connected, so the justification of Claim \ref{cl_5}
		is complete. 
	\end{proof}
	
	\begin{claim}\label{cl_6}
		$$\overline{\bigcup_{n\in \mathbb{N\cup }\{0\}} A_{n}}=A_{\mathcal{S}_{\mathcal{G}}^{A}}.$$
	\end{claim}
	
	\begin{proof}[Proof of Claim \ref{cl_6}.] Note that the attractor $A_{\mathcal{S}_{\mathcal{G}}^{A}}$ is a limit (with respect to the Vietoris topology and Hausdorff metric on $P_{b,cl}(X)$) of the sequence $(A_n)_{n\in\N}$ satisfying Claim \ref{cl_4}. Then $A_{\mathcal{S}_{\mathcal{G}}^{A}}$ is so called Kuratowski limit (the sequence $(A_n)_{n\in\N}$ is L-convergent). By the properties of such limit (see \cite[4.16]{Nad}) we obtain
			 $$A_{\mathcal{S}_{\mathcal{G}}^{A}}=\overline{\bigcup_{n=0}^\infty A_n}.$$
	\end{proof}
	
	Taking into account \cite[Problem 6, page 117]{D}, Claim \ref{cl_4} and Claim \ref{cl_5}, we
	infer that $\bigcup_{n\in \mathbb{N\cup }\{0\}}A_{n}$ is connected.
	Consequently we conclude that $\omega _{\mathcal{G}}^{A}=A_{\mathcal{S}_{
			\mathcal{G}}^{A}}\overset{\text{Claim \ref{cl_6}}}{=}\overline{\bigcup_{n\in \mathbb{
				N\cup }\{0\}}A_{n}}$ is connected.
\end{proof}

	Now we prove the implication $(i)\Rightarrow (iii)$ from the main Theorem \ref{main}.
\begin{proposition}\label{=>}
	If the graph Lipscomb space $J^\G_A$ is connected, then $\mathcal{G}=(A,E)$ is a connected graph.
\end{proposition}
\begin{proof}
	We shall reason by contradiction. We show that, if the graph $\G=(A,E)$ is not connected graph then the quotient space $N(A)/_\Gsim=J^\G_A$ is disconnected. Note that the topology on $N(A)$ is generated by the basic sets 
	$$U_s=\{x\in N(A)\colon x\rest n=s\},$$
	where $x\rest n = x_1x_2\dots x_n$. Thus the set $U_s$ contains all infinite words beginning with the given finite word $s$ of the length $n$.
	\newline
	If $\G$ is a disconnected graph then there exists a partition $S,T$ of $A$ such that  $\{s,t\}\notin E$ for every $s\in S$ and $t\in T$. Then the clopen sets 
	$$S'=\bigcup_{s\in S} U_s \quad\text{ and }\quad T'=\bigcup_{t\in T} U_t$$
	form a partition of $N(A)$. For every $x,y\in N(A)$ we show that if $x~\Gsim~y$  then $x,y\in S'$ or $x,y\in T'$. Indeed, if $x~\Gsim~y$ then their first letters are adjacent: $\{x_1,y_1\}\in E$. This means that $x_1,y_1\in S$ or $x_1,y_1\in T$, which ends the proof. 
\end{proof}

Combining Theorem \ref{s_hom}, Proposition \ref{<=} and Proposition \ref{=>}, we get the proof of the main Theorem \ref{main}.

\section{Inverse limit of a sequence of finite graphs}

Let $A$ be a finite nonempty set. Then the space $N(A)$ can be present as the inverse limit of a sequence of finite graphs on $A^n$. This approach helps to obtain additional results about the graph Liscomb space $J_A^\G = N(A)/_{\Gsim}$.  First we present basic fact connected to graph morphisms.

We set the following notation for sets of vertices and edges: $\G=(V,E)=(V_\G,E_\G)$, $\G'=(V',E')$, $\G_n=(V_n,E_n)$. 
\begin{definition}
	A \textit{homomorphism} of the graphs $\G$ and $\G'$ is a function $f\colon V\to V'$ such that for every $u,v\in V$
	$$\{u,v\}\in E \quad \Rightarrow \quad \{f(u),f(v)\}\in E'.$$
\end{definition}

\begin{definition}
	A \emph{proper epimorphism} is a mapping $\map f {V}{V'}$ such that $f(V)=V'$ and for every $x',y' \in V'$ the following equivalence holds:
	$$\{x',y'\}\in E' \iff \text{ there exists }x,y \in V \text{ such that } \{x,y\}\in E \text{ and } x' = f(x) \text{ and } y' = f(y).$$
\end{definition}
In particular, a proper epimorphism is a surjective graph homomorphism.

Proper epimorphisms have a very nice property:
\begin{lemma}
	Given two compatible graph epimorphisms, $f \cmp g$ is proper $\implies$ $f$ is proper.
\end{lemma}

\begin{proof}
	Suppose $f \cmp g$ is proper and $f(g(x))$ is adjacent to $f(g(y))$. Then there are $x', y'$ such that  $f(g(x')) = f(g(x))$, $f(g(y')) = f(g(y))$ and $x'$ is adjacent to $y'$. Finally, $g(x')$ and $g(y')$ are adjacent, because $g$ is a homomorphism.
\end{proof}

We shall consider inverse limits of sequences of finite graphs with proper epimorphisms.

An inverse sequence is fully described by a sequence $(\G_n)_{n\in\N}$ together with bonding mappings $\map{g^m_n}{V_m}{V_n}$ satisfying the natural conditions, making it a contravariant functor from the natural numbers. Its limit is simply the space of all branches of the tree induced by the sequence, namely all sequences of the form $(x_n)_{n\in\N}$, where $x_n = g^m_n(x_m)$ for every $n < m$. This also gives an idea of a natural topology, namely, a basic open set is of the form
$$U_s = \{x \in V_\G\colon x \rest n = s\},$$ 
where $s$ is a finite initial branch of length $n$ and $V_\G$ denotes the branch space.

Whenever all $\G_n$'s are finite, $\G$ is compact and zero-dimensional (also called \emph{totally disconnected}). 

Given an inverse sequence of finite graphs $(\G_n)_{n\in\N}$ with bonding mappings $g^m_n$ ($m>n$), its limit carries a natural closed graph adjacency relation, defined by
$$x \sim y \iff \forall\; n\in\N \; \{x_n ,y_n\}\in E_n.$$
{This limit is also a graph $\G=(V_\G,\sim)$, where the set of edges is exceptionally denoted by $\sim$.} Recall that we consider simple graphs with loops everywhere, namely, structures with a symmetric reflexive binary relation.

In any case, a sequence as above induces a closed graph relation on its limit.
It turns out that the converse holds:

\begin{proposition}
	Let $K$ be a metrizable zero-dimensional space, endowed with a closed graph relation. Then there is an inverse sequence of finite graphs with proper epimorphisms whose limit is $K$.
\end{proposition}

\begin{proof}
	(Sketch) We take a fixed sequence of finite spaces with surjections, whose limit is $K$. Every element of this sequence corresponds to a partition of $K$ into clopen sets. We declare an edge between two elements $U,V$ of a fixed partition if and only if there are $u \in U$, $v \in V$ with $u \sim v$.
	By this way we turn our sequence into graphs with proper epimorphisms.
\end{proof}

\begin{lemma}\label{LMniceonessxxs}
Assume $\map f {V_\G}{V_\H}$ is a proper epimorphism of graphs such that $\H$ is connected and all $f$-fibers (induced subgraphs of the form $f^{-1}(x)$, $x\in V_\H$) are connected. Then $\G$ is connected.
\end{lemma}

\begin{proof}
Suppose $V_\G = A \cup B$, where no $a \in A$ is adjacent to any $b \in B$. Then for each $x\in V_\H$ the fiber $f^{-1}(x)$ is contained either in $A$ or in $B$.
In particular, we have a decomposition $V_\H = A' \cup B'$ such that
$$A = \bigcup_{x \in A'}f^{-1}(x) \quad\text{ and }\quad B = \bigcup_{x \in B'}f^{-1}(x).$$
Since $\H$ is connected, there are $a' \in A'$ and $b' \in B'$ such that $\{a' ,b'\}\in E_\H$. Since $f$ is proper, there are $a \in f^{-1}(a') \subs A$ and $b \in f^{-1}(b) \subs B$ such that $\{a, b\}\in E_\G$. This is a contradiction.
\end{proof}

\begin{lemma}\label{LMDwanbueo}
Let $\G=(V_\G,\sim)$ be the inverse limit of connected finite graphs with proper epimorphisms.
Then for every clopen partition $V_\G = U \cup W$ there exist $u \in U$, $w\in W$ with $u \sim w$.
\end{lemma}

\begin{proof}
Let $g^\infty_n$ denote the canonical projection from $V_\G$ to $V_n$, where $(\G_n)_{n\in\N}$ is the sequence of finite connected graphs whose limit is $\G$. Assuming $U, W$ form a clopen partition of $V_\G$, there is $n$ such that $U = (g_n^\infty)^{-1}(U')$, $W = (g_n^\infty)^{-1}(W')$, where $U', W'$ form a partition of $V_{n}$.

Since $\G_n$ is connected, there are $x \in U'$, $y \in W'$ such that $\{x,y\}\in E_{n}$.
The mapping $g^\infty_n$ is a proper epimorphism, therefore there are $u,w$ such that $g^\infty_n(u)=x$, $g^\infty_n(w)=y$ and $u \sim w$. Clearly, $u \in U$ and $w\in W$.	
\end{proof}

\begin{theorem}
Assume $\G$ is the inverse limit of a sequence $(\G_n)_{n\in\N}$ of finite graphs and the graph relation $\sim$ on $\G$ is transitive.
Then the topological quotient $\G/_\sim$ is connected if and only if all graphs $\G_n$ are connected.
\end{theorem}

\begin{proof}
The ``if'' part follows from Lemma~\ref{LMDwanbueo}, because if $\G /_\sim$ was disconnected, we would have a clopen partition $U \cup W = V_\G$ such that $U$, $W$ are unions of $\sim$-equivalence classes.

The ``only if'' part is clear: Say, if $\G_k$ is disconnected, we obtain a clopen partition of the form $(g_k^\infty)^{-1}(U')\cup (g_k^\infty)^{-1}(W') = V_\G$, where $U'$, $W'$ are nonempty unions of some components of $V_{k}$. Obviously, there can be no adjacencies between pairs of elements from this partition.
\end{proof}

\begin{theorem}
Assume $\G$ is the inverse limit of finite graphs $(\G_n)_{n\in\N}$ with proper epimorphisms, such that the graph relation $\sim$ on $\G$ is transitive.
Then the topological quotient space $\G/_\sim$ is locally connected if and only if there exists $n_0$ such that for each $n \geq n_0$ the bonding mapping from $\G_{n+1}$ onto $\G_n$ has connected fibers.
\end{theorem}

\begin{proof}
Suppose all the bonding mappings above some $n_0$ have connected fibers. Removing an initial part of the sequence, we may assume $n_0 = 0$.
Fix $k$ and fix $v \in V_k$. The basic clopen set
$$U_v := \{x \in V_\G\colon g^\infty_k(x)=v\}$$
is connected, by Lemma~\ref{LMniceonessxxs}.
Hence, $\G$ is locally connected.

Now suppose for every $n$ there is $k_n>n$ such that the bonding mapping $\map{g^{k_n}_n}{V_{k_n}}{V_n}$ has a disconnected fiber.
Then we can construct a branch through the tree induced by our inverse sequence, so that infinitely many neighborhoods of it are disconnected, therefore $\G$ is not locally connected at this branch. 
\end{proof}

\begin{corollary}\label{col_tran}
Assume $\G$ is the inverse limit of finite graphs $(\G_n)_{n\in\N}$ with proper epimorphisms, such that $|V_{1}|=1$ and for each $n$ the bonding mapping from $\G_{n+1}$ onto $\G_n$ has connected fibers. If the graph relation $\sim$ on $\G$ is transitive, then the quotient topological space $\G/_ \sim$ is connected and locally connected.
\end{corollary}

Now it is time to apply this corollary to the graph Lipscomb's space $J^\G_A = N(A)/_\Gsim$ for a finite graph $\G$.

\begin{corollary}
	For a finite graph $\G=(A,E)$ the following are equivalent: 
	\begin{enumerate}[i)]
		\item $\G$ is a connected graph;
		\item the graph Lipscomb's space $J^\G_A$ (and the space $\w^A_\G$) is connected and locally connected;
		\item the graph Lipscomb's space $J^\G_A$ (and the space $\w^A_\G$) is arcwise connected.	
	\end{enumerate}
\end{corollary}
\begin{proof}
	The spaces $J^\G_A$ and $\w^A_\G$ are homeomorphic by Theorem \ref{s_hom} so they are simultaneously connected, locally connected and arcwise connected.\\
	For the implication $i)\Rightarrow ii)$ note that $N(A)$ is the inverse limit of a sequence of finite graphs $\G_n=(A^n,E_n)$ for $n\in\N\cup\{0\}$, whose bonding mappings are proper epimorphisms, $A^0= \{\emptyset\}$, and for each $n$ the bonding mapping from $A^{n+1}$ to $A^n$ has fibers isomorphic to $\G=(A,E)$. The relation $\Gsim$ is transitive (as an equivalence relation) so by Corollary \ref{col_tran} we obtain that $J^\G_A = N(A)/_\Gsim$ is connected and locally connected.\\
	The implication $ii)\Rightarrow iii)$ follows from the fact that $J^\G_A$ is compact (as $A$ is finite) and metrizable, and every locally connected metrizable continuum is pathwise and arcwise connected (see \cite[6.3.11]{Eng}).\\
	The implication $iii)\Rightarrow i)$ follows from Theorem \ref{main}.
\end{proof}

It seems natural to generalize the construction from this section, at least by varying the graph structure at each moment of time. Of course, one needs to think of a natural embedding of the quotient space into a Euclidean space, then analyzing whether it is an IFS-attractor.

Making things too general will lead to pathological examples, as every nonempty compact metrizable space is a quotient of the Cantor set and every $1$-dimensional compact metrizable space (e.g. the pseudoarc, the shark teeth, the Warsaw circle, etc.) is a 2-1 image of the Cantor set.

\end{document}